\documentclass[a4paper,12pt,reqno,twoside]{amsart}

\usepackage{amssymb,amsmath}
\usepackage[foot]{amsaddr}
\usepackage[colorlinks,citecolor=blue,urlcolor=blue]{hyperref}
\usepackage[hmargin=3.5cm,vmargin=2.5cm]{geometry}
\usepackage[inline]{enumitem}

\usepackage{tikz}
\usetikzlibrary{automata,arrows,positioning,calc}
\usepackage{lipsum}
\usepackage[normalem]{ulem}
\newcommand{\eps}{\varepsilon}

\newcommand{\Lip}{{\mathrm{Lip}}}

\newcommand\R{{\mathbb R}}

\newcommand\cD{{\mathcal D}}

\newcommand\cL{{\mathcal L}}

\newcommand\bfC{{\mathbf C}}

\newcommand{\tcD}{{\widetilde{\cD}}}

\newcommand{\tA}{{\tilde A}}
\newcommand{\ta}{{\tilde a}}
\newcommand{\tf}{{\tilde f}}

\newcommand{\tgamma}{{\tilde{\gamma}}}

\numberwithin{equation}{section}

\newtheorem{thm}{Theorem}[section]
\newtheorem{prop}[thm]{Proposition}
\newtheorem{cor}[thm]{Corollary}
\newtheorem{lem}[thm]{Lemma}

\theoremstyle{remark}
\newtheorem{rmk}[thm]{Remark}

\begin{document}

\title{Statistical aspects of mean field coupled intermittent maps}
\author{Wael Bahsoun}
\address{Department of Mathematical Sciences, Loughborough University,
Loughborough, Leicestershire, LE11 3TU, UK}
\email{$\dagger$ W.Bahsoun@lboro.ac.uk, $\ddagger$ a.korepanov@lboro.ac.uk}
\author{Alexey Korepanov}
\date{23 June 2023}
\thanks{The research of both authors is supported by EPSRC grant EP/V053493/1.
A.K.~is thankful to Nicholas Fleming-V\'azquez for helpful discussions and advice.
}
\maketitle

\begin{abstract}
    We study infinite systems of mean field weakly coupled intermittent maps
    in the Pomeau-Manneville scenario.
    We prove that the coupled system admits a unique ``physical'' stationary state,
    to which all absolutely continuous states converge.
    Moreover, we show that suitably regular states converge polynomially.
\end{abstract}

\section{Introduction}

Mean field coupled dynamics can be thought of as a dynamical system
with $n$ ``particles'' with states $x_1, \ldots, x_n$
evolving according to an equation of the type
\[
    x_k \mapsto T \Bigl( x_k, \eps \frac{\delta_{x_1} + \cdots + \delta_{x_n}}{n} \Bigr)
    .
\]
Here $T$ is some transformation, $\eps \in \R$ is the strength of coupling
and $\delta_{x_k}$ are the delta functions, so $(\delta_{x_1} + \cdots + \delta_{x_n}) / n$
is a probability measure describing the ``mean state'' of the system.

As $n \to \infty$, it is natural to consider
the evolution of the distribution of particles:
if $\mu$ is a probability measure describing distribution of particles, then
one looks at the operator that maps $\mu$ to the distribution of
$T(x, \eps \mu)$, where $x \sim \mu$ is random.

In chaotic dynamics, mean field coupled systems have been studied first when $T$ is
a perturbation of a uniformly expanding circle map by Keller~\cite{K00} and followed,
among others, by B\'alint, Keller, S\'elley and T\'oth~\cite{BKST18},
Blank~\cite{B11}, Galatolo~\cite{G21}, S\'elley and Tanzi~\cite{ST21}.
The case when $T$ is a perturbation of an Anosov diffeomorphism has been
covered by Bahsoun, Liverani and S\'elley~\cite{BLS22} (see in particular~\cite[Section~2.2]{BLS22}
for a motivation of such study).
See Galatolo~\cite{G21} for a general framework when the site dynamics admits exponential
decay of correlations.
The results of~\cite{G21} also apply to certain mean field coupled random systems.
We refer the reader to Tanzi~\cite{T22} for a recent review on the topic and to~\cite{BLS22}
for connections with classical and important partial differential equations.

In this work we consider the situation where $T$ is a perturbation of the prototypical
chaotic map with \emph{non-uniform} expansion and polynomial decay of correlations:
the intermittent map on the unit interval $[0,1]$ in the Pomeau-Manneville scenario~\cite{PM80}.
We restrict to the case when the coupling is weak, i.e.\ $\eps$ is small.

Our results apply to a wide class of intermittent systems satisfying standard assumptions
(see Section~\ref{sec:main}).
To keep the introduction simple, here we consider a very concrete example.

Fix $\gamma_* \in (0,1)$ and let, for $\eps \in \R$, $h \in L^1[0,1]$ and $x \in [0,1]$,
\begin{equation}
    \label{eq:Teh}
    T_{\eps h} (x) = x (1 + x^{\gamma_* + \eps \gamma_h}) + \eps \varphi_h(x) \mod 1,
\end{equation}
where
\[
    \gamma_h = \int_0^1 h(s) \sin (2 \pi s) \, ds
    \quad \text{and} \quad
    \varphi_h(x) = x^2 (1-x) \int_0^1 h(s) \cos (2 \pi s) \, ds
    .
\]
This way, $T_{\eps h}$ is a perturbation of the intermittent map $x \mapsto x (1 + x^{\gamma_*}) \mod 1$.
Informally, $\gamma_h$ changes the degree of the indifferent point at $0$, and $\varphi_h$
is responsible for perturbations away from $0$.

We restrict to $\eps \in [-\eps_0, \eps_0]$ with $\eps_0$ small
and to $h$ nonnegative with $\int_0^1 h(x) \, dx = 1$
(i.e.\ $h$ is a probability density).

Let $\cL_{\eps h} \colon L^1[0,1] \to L^1[0,1]$ be the transfer operator for $T_{\eps h}$:
\begin{equation}
    \label{eq:tr}
    (\cL_{\eps h} g) (x)
    = \sum_{y \in T_{\eps h}^{-1}(x)} \frac{g(y)}{T_{\eps h}'(y)}
    ,
\end{equation}
and let
\begin{equation} \label{eq:selfcL}
    \cL_{\eps} h = \cL_{\eps h} h
    .
\end{equation}
We call $\cL_\eps$ the \emph{self-consistent} transfer operator.
Observe that $\cL_{\eps}$ is nonlinear, and that $\cL_\eps h$ is the density
of the distribution of $T_{\eps h} (x)$, if $x$ is distributed according
to the probability measure with density $h$.

We prove that for sufficiently small $\eps_0$,
the self-consistent transfer operator $\cL_\eps$ admits a unique \emph{physical}
(see~\cite[Definition~2.1]{BLS22}) invariant state $h_\eps$, and that
$\cL_\eps^n h$ converges to $h_\eps$ in $L^1$ polynomially for all sufficiently regular $h$:

\begin{thm}
    \label{thm:example-ac}
    There exists $\eps_0 \in (0, 1 - \gamma_*)$ so that
    each $\cL_\eps$ with $\eps \in [-\eps_0, \eps_0]$, as an operator on probability densities,
    has a unique fixed point $h_\eps$. For every probability density $h$,
    \[
        \lim_{n \to \infty} \| \cL_\eps^n h - h_\eps \|_{L^1}
        = 0
        .
    \]
    Moreover, $h_\eps \in C^\infty (0,1]$ and there are $A, a_1, a_2, \ldots > 0$ such that
    for all $\ell \geq 1$ and $x \in (0,1]$,
    \begin{equation}
        \label{eq:Aaaa}
        \int_0^x h_\eps(s) \, ds
        \leq A x^{1 - 1 / (\gamma_* + \eps_0)}
        \quad \text{and} \quad
        \frac{|h_\eps^{(\ell)}(x)|}{h_\eps(x)}
        \leq \frac{a_\ell}{x^\ell}
        .
    \end{equation}
\end{thm}

\begin{thm}
    \label{thm:example-mix}
    In the setup of Theorem~\ref{thm:example-ac},
    suppose that a probability density $h$ is twice differentiable on $(0,1]$ and satisfies,
    for some $\tA, \ta_1, \ta_2 > 0$ and all $\ell = 1,2$ and $x \in (0,1]$,
    \[
        \int_0^x h(s) \, ds
        \leq \tA x^{1 - 1 / (\gamma_* + \eps_0)}
        \quad \text{and} \quad
        \frac{|h_\eps^{(\ell)}(x)|}{h_\eps(x)}
        \leq \frac{\ta_\ell}{x^\ell}
        .
    \]
    Then
    \begin{equation}
        \label{eq:Ln}
        \| \cL_\eps^n h - h_\eps \|_{L^1}
        \leq C n^{ - (1 - \gamma_* - \eps_0) / (\gamma_* + \eps_0)}
        ,
    \end{equation}
    where $C$ depends only on $\tA, \ta_1, \ta_2$ and $\eps_0$.
\end{thm}

\begin{rmk}
    The restriction $\eps_0 < 1 - \gamma_*$ serves to guarantee that $\gamma_* + \eps \gamma_h$
    is bounded away from $1$, and that the right hand side of~\eqref{eq:Ln} converges to zero.
\end{rmk}

\begin{rmk}
    A curious corollary of Theorem~\ref{thm:example-ac} is that
    the density of the unique absolutely continuous invariant probability measure
    for the map $x \mapsto x ( 1 + x^\gamma_*)$ is smooth, namely $C^\infty(0,1]$
    with the bounds~\eqref{eq:Aaaa}. Our abstract framework covers such a result also for
    the Liverani-Saussol-Vaienti maps~\cite{LSV99}.
    To the best of our knowledge, this is the first time such a result is written down.
    At the same time, we are aware of at least two different unwritten prior proofs
    which achieve similar or stronger results,
    one by Damien Thomine and the other by Caroline Wormell.
\end{rmk}

\begin{rmk}
    Another example to which our results apply is
    \[
        T_{\eps h} (x)
        = x (1 + x^{\gamma_*})+ \eps x (1-x) \int_0^1h(s) \sin(\pi s) \, ds \mod 1
        ,
    \]
    where $\gamma_* \in (0,1)$ and $\eps \in [0, \eps_0]$.
    This is interesting because now each $T_{\eps h}$ with $\eps > 0$ is uniformly
    expanding, but the expansion is not uniform in $\eps$.
    Thus, even for this example, standard operator contraction techniques
    employed in~\cite{G21,K00} do not apply.
\end{rmk}

\begin{rmk}
    Let $h_\eps$ be as in Theorem \ref{thm:example-ac}.
    A natural question is to study the regularity of the map $\eps \mapsto h_\eps$.
    We expect that it should be differentiable in a suitable topology.
\end{rmk}

The paper is organised as follows.
Theorems~\ref{thm:example-ac} and~\ref{thm:example-mix} are corollaries
of the general results in Section~\ref{sec:main},
where we introduce the abstract framework and state the abstract results.
The abstract proofs are carried out in Section~\ref{sec:proofs},
and in Section~\ref{sec:example} we verify that the specific map~\eqref{eq:Teh}
fits the abstract assumptions.

\section{Assumptions and results}
\label{sec:main}

We consider a family of maps $T_{\eps h} \colon [0,1] \to [0,1]$,
where $\eps \in [-\eps_*, \eps_*]$, $\eps_* > 0$, and $h$ is a probability density on $[0,1]$.

We require that each such $T_{\eps h}$ 
is a full branch increasing map with finitely many branches, i.e.\
there is a finite partition of the interval $(0,1)$
into open intervals $B_{\eps h}^k$, modulo their endpoints,
such that each restriction $T_{\eps h} \colon B_{\eps h}^k \to (0,1)$ is an increasing bijection.

We assume that each restriction $T_{\eps h} \colon B_{\eps h}^k \to (0,1)$ satisfies the following assumptions
with the constants independent of $\eps$, $h$ or the branch:
\begin{enumerate}[label=(\alph*)]
    \item\label{a:r}
        $T_{\eps h}$ is $r+1$ times continuously differentiable with $r \geq 2$.
    \item\label{a:gamma}
        There are $c_\gamma > 0$, $C_\gamma > 1$ and $\gamma \in [0,1)$ such that
        \begin{equation}
            \label{eq:tailgamma}
            1 + c_\gamma x^\gamma
            \leq T_{\eps h}'(x)
            \leq C_\gamma
            .
        \end{equation}
    \item\label{a:w}
        Denote $w = 1 / T'_{\eps h}$. There are $b_1, \ldots, b_r > 0$ and $\chi_* \in (0,1]$ so that
        for all $1 \leq \ell \leq r$, $0 \leq j \leq \ell$
        and each monomial $w_{\ell,j}$ in the
        expansion of $(w^\ell)^{(\ell-j)}$,
        \begin{equation}
            \label{eq:bbb}
            \frac{w^\ell}{\chi_\ell}
            \leq \frac{1}{\chi_\ell \circ T_{\eps h}} - b_\ell \frac{|w_{\ell,j}|}{\chi_j}
            ,
        \end{equation}
        where $\chi_\ell(x) = \min \{ x^\ell, \chi_* \}$.
        (For example, the expansion of $(w^3)''$ is $6w (w')^2 + 3 w^2 w''$.)
    \item\label{a:d}
        If $\partial B^k_{\eps h} \not \ni 0$, i.e.\ $B^k_{\eps h}$ is not the leftmost branch, then $T_{\eps h}$
        has bounded distortion:
        \begin{equation}
            \label{eq:dist2}
            \frac{T_{\eps h}''}{(T_{\eps h}')^2}
            \leq C_d
            ,
        \end{equation}
        with $C_d > 0$.
\end{enumerate}

\begin{rmk}
    Assumption~\ref{a:w} is unusual, but we did not see a way to replace it with something natural.
    At the same time, it is straightforward to verify and to apply.
    It plays the role of a distortion bound in $C^r$ adapted to an intermittency at $0$.
\end{rmk}

In addition to the above, we assume that the transfer operators corresponding to $T_{\eps h}$
vary nicely in $h$.
We state this formally in~\eqref{eq:Ldiffgamma}, after we introduce the required notation.

Define the transfer operators $\cL_{\eps h}$ and $\cL_\eps$ as in~\eqref{eq:tr}
and~\eqref{eq:selfcL}.

For an integer $k \geq 1$, let $H^k$ denote the set of $k$-H\"older functions
$g \colon (0,1] \to (0, \infty)$, i.e.\ such that $g$ is $k - 1$ times continuously
differentiable with $g^{(k-1)}$ Lipschitz.
Denote $\Lip_g (x) = \limsup_{y \to x} |g(x) - g(y)| \big/ |x - y|$.

Suppose that $a_1, \ldots, a_r > 0$. For $1 \leq k \leq r$, let
\begin{equation}
    \label{eq:cCr}
    \begin{aligned}
        \cD^k
        = \Bigl\{
            g \in H^k
            : \ &
            \frac{|g^{(\ell)}|}{g} \leq \frac{a_\ell}{\chi_\ell}
            \ \ \text{for all} \ \ 1 \leq \ell < k,
            \\ &
            \frac{\Lip_{g^{(k-1)}}}{g} \leq \frac{a_k}{\chi_k}
        \Bigr\}
        .
    \end{aligned}
\end{equation}
Take $A > 0$ and let
\begin{equation}
    \label{eq:cCr1}
    \cD^k_1
    = \Bigl\{
        g \in \cD^k :
        \textstyle \int_0^1 g(s) \, ds = 1
        , \ 
        \textstyle \int_0^x g(s) \, ds \leq A x^{1 - \gamma}
    \Bigr\}
    .
\end{equation}

\begin{rmk}
    \label{rmk:xgamma}
    If $g \in \cD^1_1$ then $g(x) \leq C x^{-\gamma}$ where
    $C$ depends only on $a_1$ and $A$.
\end{rmk}

Now and for the rest of the paper we fix $a_1, \ldots, a_r$ and $A$ so that $\cD^k$ and $\cD_1^k$
are non-empty and invariant under $\cL_{\eps h}$. This can be done thanks to the following lemma.

\begin{lem}
    \label{lem:Cr}
    There are $a_1, \ldots, a_r, A > 0$ such that for all $1 \leq q \leq r$,
    \begin{enumerate}[label=(\alph*)]
        \item\label{lem:Cr:d} $g \in \cD^q$ implies $\cL_{\eps h} g \in \cD^q$.
        \item\label{lem:Cr:1} $g \in \cD^q_1$ implies $\cL_{\eps h} g \in \cD^q_1$.
    \end{enumerate}
    Moreover, given $C > 0$ we can ensure that
    $\min\{a_1, \ldots, a_r, A \} > C$.
\end{lem}

The proof of Lemma~\ref{lem:Cr} is postponed to Section~\ref{sec:proofs}.

Finally, we assume that there are $0 \leq \beta < \min\{ \gamma, 1 - \gamma \}$
and $C_\beta > 0$ such that if $h_0, h_1 \in L^1$
and $v \in \cD^2_1$, then
\begin{equation}
    \label{eq:Ldiffgamma}
    \cL_{\eps h_0} v - \cL_{\eps h_1} v
    = \delta (f_0 - f_1),
\end{equation}
for some $f_0, f_1 \in \cD^1_1$ with $f_0(x), f_1(x) \leq C_\beta x^{-\beta}$
and $\delta \leq |\eps| C_\beta \|h_0 - h_1\|_{L^1}$.

Let $\bfC = (\eps_*, r, c_\gamma, C_\gamma, \gamma, b_1, \ldots, b_r, \chi_*, C_d,
A, a_1, \ldots, a_r, \beta, C_\beta)$
be the collection of constants from the above assumptions.

Our main abstract result is the following theorem:

\begin{thm}
    \label{thm:acim}
    There exists $\eps_0 > 0$ such that for every $\eps \in [-\eps_0, \eps_0]$:
    \begin{enumerate}[label=(\alph*)]
        \item\label{thm:acim:h}
            There exists $h_\eps$ in $\cD^r_1$ so that for every probability density $h$,
            \[
                \lim_{n \to \infty} \| \cL_\eps^n h - h_\eps \|_{L^1}
                = 0
                .
            \]
        \item\label{thm:acim:D}
            Let $\tcD^2_1$ be a version of $\cD^2_1$ with constants $\tA, \ta_1, \ta_2$ in place of 
            $A, a_1, a_2$. (We do not require that $\tcD^2_1$ is invariant.)
            Then for every $h \in \tcD^2_1$,
            \[
                \| \cL_\eps^n h - h_\eps \|_{L^1}
                \leq C n^{1-1/\gamma}
                ,
            \]
            where $C$ depends only on $\bfC$ and $\tA, \ta_1, \ta_2$.
    \end{enumerate}
\end{thm}

\section{Proofs}
\label{sec:proofs}

In this section we prove Lemma~\ref{lem:Cr} and Theorem~\ref{thm:acim}.
The latter follows from Lemma~\ref{lem:fix} and Propositions~\ref{prop:mix}, \ref{prop:physical}.

Throughout we work with maps $T_{\eps h}$ as per our assumptions, in particular $\eps$ is always
assumed to belong to $[-\eps_*, \eps_*]$, and $h$ is always a probability density.

\subsection{Invariance of \texorpdfstring{$\cD^q$, $\cD^q_1$}{cones} and distortion bounds}

We start with the proof of Lemma~\ref{lem:Cr}. Our construction of $A, a_1, \ldots, a_r$
allows them to be arbitrarily large, and without mentioning this further,
we restrict the choice so that
\begin{equation}
    \label{eq:xgamma}
    x \mapsto (1 - \gamma) x^{-\gamma}
    \quad \text{is in} \quad
    \breve{\cD}_1^r
    ,
\end{equation}
where $\breve{\cD}_1^r$ is the version of $\cD_1^r$ with $A/2, a_1 / 2, \ldots, a_r / 2$ in place
of $A, a_1, \ldots, a_r$.
Informally, we require that $(1-\gamma) x^{-\gamma}$ is deep inside $\cD^r_1$.

\begin{lem}
    \label{lem:branch}
    There is a choice of $a_1, \ldots, a_r$ such that
    if $B \subset (0,1)$ is a branch of $T_{\eps h}$ and $g \in \cD^q$ with $1 \leq q \leq r$,
    then $\cL_{\eps h} (1_B g) \in \cD^q$.
\end{lem}

\begin{proof}
    To simplify the notation, let $T \colon B \to (0,1)$
    denote the restriction of $T_{\eps h}$ to $B$. Then its inverse $T^{-1}$ is well defined.
    Let $w = 1 / T'$ and $f = (g w) \circ T^{-1}$.
    We have to choose $a_1, \ldots, a_r$ show that $f \in \cD^q$ independently of $g$ and $B$.

    For illustration, it is helpful to write out a couple of derivatives of $f$: 
    \begin{align*}
        f' & = [g' w^2 + g w' w] \circ T^{-1}
        , \\
        f'' & = [g'' w^3 + 3 g' w' w^2 + g w'' w^2 + g (w')^2 w] \circ T^{-1}
        .
    \end{align*}
    An observation that $f^{(\ell)} = (u_\ell w) \circ T^{-1}$, where
    $u_0 = g$ and $u_{\ell + 1} = (u_\ell w)'$, generalizes the pattern:
    \begin{equation}
        \label{eq:bnau}
        f^{(\ell)}
        = \Bigl[
            g^{(\ell)} w^{\ell + 1} + \sum_{j = 0}^{\ell - 1} g^{(j)} W_{\ell,j} w
        \Bigr] \circ T^{-1}
        .
    \end{equation}
    Here each $W_{\ell,j}$ is a linear combination of monomials from the expansion of
    $(w^\ell)^{(\ell-j)}$.

    By~\eqref{eq:bbb}, for each $\ell$ there is $c_\ell > 0$, depending only on $b_1, \ldots, b_{\ell-1}$, such that
    \[
        \frac{w^\ell}{\chi_\ell}
        \leq \frac{1}{\chi_\ell \circ T}
        - c_\ell \sum_{j=0}^{\ell-1} \frac{|W_{\ell,j}|}{\chi_j}
        .
    \]
    Using this and the triangle inequality,
    \begin{equation}
        \label{eq:horror}
        \begin{aligned}
            \Bigl|
            \frac{g^{(\ell)}}{g} w^{\ell}
        & + \sum_{j = 0}^{\ell - 1} \frac{g^{(j)}}{g} W_{\ell,j}
        \Bigr|
        \leq \frac{|\chi_\ell g^{(\ell)}|}{g} \frac{w^{\ell}}{\chi_\ell}
        + \max_{j < \ell} \frac{|\chi_j g^{(j)}|}{g} \sum_{j = 0}^{\ell - 1} \frac{|W_{\ell,j}|}{\chi_j}
        \\
        & \leq \frac{|\chi_\ell g^{(\ell)}|}{\chi_\ell \circ T \, g}
        - \Bigl[
            c_\ell \frac{|\chi_\ell g^{(\ell)}|}{g}
            - \max_{j < \ell} \frac{|\chi_j g^{(j)}|}{g}
        \Bigr]
        \sum_{j = 0}^{\ell - 1} \frac{|W_{\ell,j}|}{\chi_j}
        .
        \end{aligned}
    \end{equation}
    Choose $a_1 \geq c_1^{-1}$ and $a_\ell \geq c_\ell^{-1} \max_{j < \ell} a_j$ for $2 \leq \ell \leq r$.
    It is immediate that if $g \in \cD^q$ and $1 \leq \ell < q$, then
    the right hand side of~\eqref{eq:horror} is at most $a_\ell / \chi_\ell \circ T$,
    which in turn implies that $f^{(\ell)} / f \leq a_\ell / \chi_\ell$.
    A similar argument yields $\Lip_{f^{(q-1)}} / f \leq a_q / \chi_q$,
    hence $f \in \cD^q$ as required.
\end{proof}

\begin{proof}[Proof of Lemma~\ref{lem:Cr}]
    First we show that part~\ref{lem:Cr:d} follows from Lemma~\ref{lem:branch}.
    Indeed, let $a_1, \ldots, a_r$ be as in Lemma~\ref{lem:branch} and suppose that
    $g \in \cD^q$. Write
    \[
        \cL_{\eps h} g
        = \sum_B \cL_{\eps h} (1_B g)
        ,
    \]
    where the sum is taken over the branches of $T_{\eps h}$.
    Each $\cL_{\eps h} (1_B g)$ belongs to $\cD^q$ by Lemma~\ref{lem:branch},
    and $\cD^q$ is closed under addition. Hence $\cL_{\eps h} g \in \cD^q$.

    It remains to prove part~\ref{lem:Cr:1} by choosing a suitable $A$.
    Without loss of generality, we restrict to $q = 1$.

    Fix $\eps$, $h$ and denote, to simplify notation, $T = T_{\eps h}$ and $\cL = \cL_{\eps h}$.
    Suppose that $g \in \cD^1$ with $\int_0^1 g(s) \, ds = 1$ and
    $\int_0^x g(s) \, ds \leq A x^{1-\gamma}$ for all $x$.
    We have to show that if $A$ is sufficiently large, then
    $\int_0^x (\cL g)(s) \, ds \leq A x^{1-\gamma}$.
    
    Suppose that $T$ has branches $B_1, \ldots, B_N$, where $B_1$ is the leftmost branch.
    Denote by $T_k \colon B_k \to (0,1)$ the corresponding restrictions.
    Taking the sum over branches, write
    \begin{equation}
        \label{eq:k}
        \int_0^x (\cL g)(s) \, ds
        = \sum_{k = 1}^N \int_{T_k^{-1} (0,x)} g(s) \, ds
        .
    \end{equation}
    Since $T_1^{-1}(x) \leq x / (1 + c x^\gamma)$ with some $c$ depending only on $c_\gamma$ and $\gamma$,
    \begin{equation}
        \label{eq:k1}
        \int_{T_1^{-1} (0,x)} g(s) \, ds
        \leq A \biggl(\frac{x}{1 + c x^\gamma}\biggr)^{1-\gamma}
        \leq A ( x^{1-\gamma} - c' x)
        ,
    \end{equation}
    where $c' > 0$ also depends only on $c_\gamma$ and $\gamma$.

    Let now $k \geq 2$. Note that $T_k^{-1}(0,x) \subset (C_\gamma^{-1}, 1)$.
    Observe that if $g \in \cD^1$ with $\int_0^1 g(s) \, ds = 1$, then
    $g(s) \leq C$ for $s \in (C_\gamma^{-1}, 1)$,
    where $C$ depends only on $a_1$ and $\chi_*$.
    Since $T_k$ is uniformly expanding with bounded distortion~\eqref{eq:dist2},
    $|T_k^{-1}(0,x)| \leq C' |B_k| x$ with some $C'$ that depends only on $C_d$.
    Hence
    \begin{equation}
        \label{eq:k2}
        \int_{T_k^{-1} (0,x)} g(s) \, ds
        \leq C C' |B_k| x
        .
    \end{equation}

    Assembling~\eqref{eq:k}, \eqref{eq:k1} and~\eqref{eq:k2}, we have
    \[
        \int_0^x (\cL g)(s) \, ds
        \leq A ( x^{1-\gamma} - c' x) + C'' x
    \]
    with $c', C'' > 0$ independent of $A$, $\eps$ and $h$.
    For each $A \geq C'' / c'$ the right hand side above is
    bounded by $A x^{1-\gamma}$, as desired.
\end{proof}

A useful corollary of Lemma~\ref{lem:branch} is a distortion bound:

\begin{lem}
    \label{lem:dist}
    Let $n > 0$ and $\delta > 0$.
    Consider maps $T_{\eps h_k}$, $1 \leq k \leq n$ with some $\eps$ and $h_k$ as per our assumptions.
    Choose and restrict to a single branch for every $T_{\eps h_k}$, so that all $T_{\eps h_k}$ are invertible
    and $T_{\eps h_k}^{-1}$ is well defined.
    Denote
    \[
        T_n = T_{\eps h_n} \circ \ldots \circ T_{\eps h_1}
        \quad \text{and} \quad
        J_n = 1 / T_n' \circ T_n^{-1}
        .
    \]
    Then
    \begin{equation}
        \label{eq:distJ}
        \frac{|J_n^{(\ell)}|}{J_n}
        \leq \frac{a_\ell}{\chi_\ell}
        \quad \text{for} \quad 1 \leq \ell < r
        , \quad \text{and} \quad
        \frac{\Lip_{J_n^{(r-1)}}}{J_n}
        \leq \frac{a_r}{\chi_r}
        .
    \end{equation}
    In particular, for every $\delta > 0$ the bounds above are uniform in $x \in [\delta,1]$.
\end{lem}

\begin{proof}
    Let $P_{\eps h_k}$ be the transfer operator for $T_{\eps h_k}$, restricted to the chosen branch:
    \[
        P_{\eps h_k} g
        = \frac{g}{T_{\eps h_k}'} \circ T_{\eps h_k}^{-1}
        .
    \]
    Denote $P_n = P_{\eps h_n} \cdots P_{\eps h_1}$.

    Let $g \equiv 1$. Clearly, $g \in \cD^r$.
    By Lemma~\ref{lem:branch}, $P_{\eps h_k} \cD^r \subset \cD^r$ and thus $P_n g \in \cD^r$.
    On the other hand, $P_n g = J_n$, and the desired result follows from the definition of $\cD^k$.
\end{proof}

\subsection{Fixed point and memory loss}

Further let $h_\eps$ be a fixed point of $\cL_\eps$ as in the following lemma;
later we will show that it is unique.

\begin{lem}\label{lem:fix}
    There exists $h_\eps \in \cD^r_1$ such that $\cL_\eps h_\eps = h_\eps$.
\end{lem}

\begin{proof}
    Suppose that $f,g \in \cD^r_1$. Write
    \[
        \| \cL_{\eps f} f - \cL_{\eps g} g \|_{L^1}
        \leq \| \cL_{\eps f} f - \cL_{\eps f} g \|_{L^1}
        + \| \cL_{\eps f} g - \cL_{\eps g} g \|_{L^1}
        .
    \]
    The first term on the right is bounded by $\|f - g\|_{L^1}$ because $\cL_{\eps f}$
    is a contraction in $L^1$.
    By~\eqref{eq:Ldiffgamma}, so is the second term, up to a multiplicative constant.
    It follows that $\cL_\eps$ is continuous in $L^1$.
    Recall that $\cL_\eps$ preserves $\cD^r_1$ and note that $\cD^r_1$ is compact
    in the $L^1$ topology.
    By the Schauder fixed point theorem, $\cL_\eps$ has a fixed point in $\cD^r_1$.
\end{proof}

Further we use the rates of memory loss for sequential dynamics from~\cite{KL21}:
\begin{thm}
    \label{thm:Juho}
    Suppose that $f,g \in \cD^1_1$ and $h_1, h_2, \ldots$ are probability densities.
    Denote $\cL_n = \cL_{\eps h_n} \cdots \cL_{\eps h_1}$.
    Then
    \[
        \| \cL_n f - \cL_n g \|_{L^1}
        \leq C_1 n^{- 1 / \gamma + 1}
        .
    \]
    More generally, if $f(x),g(x) \leq C_\gamma' x^{-\gamma'}$ with $C_\gamma' > 0$
    and $\gamma' \in [0,\gamma]$, then
    \[
        \| \cL_n f - \cL_n g \|_{L^1}
        \leq C_2 n^{- 1 / \gamma + \gamma' / \gamma}
        .
    \]
    The constant $C_1$ depends only on $\bfC$,
    and $C_2$ depends additionally on $\gamma'$, $C_\gamma'$.
\end{thm}

\begin{proof}
    In the language of~\cite{KL21}, the family $T_{\eps h_k}$
    defines a nonstationary nonuniformly expanding dynamical system.
    As a base of ``induction'' we use the whole interval $(0,1)$.
    For a return time of $x \in (0,1)$ corresponding to a sequence $T_{\eps h_k}$, $k \geq n$,
    we take the minimal $j \geq 1$ such that $T_{\eps h_k} \circ \cdots \circ T_{\eps h_{j-1}} (x)$
    belongs to one of the right branches of $T_{\eps h_j}$, i.e.\ not to the leftmost branch.
    Note that we work with the return time which is not a first return time, unlike in~\cite{KL21},
    but this is a minor issue that can be solved by extending the space where the dynamics is defined.

    It is a direct verification that our assumptions and Lemma~\ref{lem:dist}
    verify (NU:1--NU:7) in~\cite{KL21} with tail function
    $h(n) = C n^{-1/\gamma}$ with $C$ depending only on $\gamma$ and $c_\gamma$.

    Further in the language of~\cite{KL21}, functions in $\cD^1_1$ are densities
    of probability measures with a uniform tail bound
    $C n^{-1/\gamma + 1}$, where $C$ depends only on $\bfC$; each $f \in \cD^1_1$ with
    $f(x) \leq C_\gamma' x^{-\gamma'}$ is a density of a probability measure with tail bound
    $C n^{-1/\gamma + \gamma' / \gamma}$ with $C$ depending only on $\bfC$ and $\gamma', C_\gamma'$.

    In this setup, Theorem~\ref{thm:Juho} is a particular case of~\cite[Theorem~3.8 and Remark~3.9]{KL21}.
\end{proof}

Recall that, as a part of assumption~\eqref{eq:Ldiffgamma}, we fixed
$\beta \in \bigl[ 0, \min\{ \gamma, 1 - \gamma \} \bigr)$.

\begin{lem}
    \label{lem:conv}
    There is a constant $C_{\beta, \gamma} > 0$, depending only on $\beta$ and $\gamma$,
    such that if a nonnegative sequence $\delta_n$, $n \geq 0$, satisfies
    \begin{equation}
        \label{eq:conv}
        \delta_n
        \leq \xi n^{ -1 / \gamma + 1}  + \sigma \sum_{j=0}^{n-1} \delta_j ( n - j )^{ -1 / \gamma + \beta / \gamma}
        \quad \text{for all} \quad n > 0
    \end{equation}
    with some $\sigma \in (0, C_{\beta, \gamma}^{-1})$ and $\xi > 0$, then
    \[
        \delta_n \leq \max \Bigl\{ \delta_0, \frac{\xi}{ 1 - \sigma C_{\beta, \gamma} } \Bigr\} \, n^{ -1 / \gamma + 1}
        \quad \text{for all} \quad n > 0
        .
    \]
\end{lem}

\begin{proof}
    We choose $C_{\beta,\gamma}$ which makes the following inequality true for all $n$:
    \begin{equation}
        \label{eq:Cbg}
        \sum_{j=0}^{n-1} (j + 1)^{ -1 / \gamma + 1} (n - j)^{ -1 / \gamma + \beta / \gamma }
        \leq C_{\beta, \gamma} (n + 1)^{ -1 / \gamma + 1 }
        .
    \end{equation}
    Let $K = \max \{ \delta_0, \xi / ( 1 - \sigma C_{\beta, \gamma} ) \}$. Then $\delta_0 \leq K$, and
    if $\delta_j \leq K (j + 1)^{ - 1 / \gamma + 1 }$ for all $j < n$, then by~\eqref{eq:conv} and~\eqref{eq:Cbg},
    \[
        \delta_n
        \leq \bigl( \xi + \sigma C_{\beta, \gamma} K \bigr) (n + 1)^{ -1 / \gamma + 1 }
        \leq K (n + 1)^{ -1 / \gamma + 1 }
        .
    \]
    It follows by induction that this bound holds for all $n$.
\end{proof}

\begin{prop}
    \label{prop:mix}
    Let $\tcD^2_1$ be as in Theorem~\ref{thm:acim}.
    There is $\eps_0 > 0$ and $C > 0$ such that for all $f,g \in \tcD^2_1$ and $\eps \in [-\eps_0, \eps_0]$,
    \[
        \bigl\| \cL_\eps^n f - \cL_\eps^n g \bigr\|_{L^1}
        \leq C n^{-1/\gamma + 1}
        .
    \]
\end{prop}

\begin{proof}
    Without loss of generality, suppose that $g (x) = (1 - \gamma) x^{-\gamma}$.
    By~\eqref{eq:xgamma}, $g \in \cD^2_1$.
    Choose $\xi > 0$ large enough so that $(f + \xi g) / (\xi + 1) \in \cD^2_1$.
    Such $\xi$ exists by Remark~\ref{rmk:xgamma},~\eqref{eq:xgamma} and the definition of $\cD^2_1$;
    it depends only on $A, a_1, a_2$ and $\tA, \ta_1, \ta_2$.

    Denote $f_n = \cL_\eps^n f$ and $g_n = \cL_\eps^n g$.
    Write $f_n - g_n = A_n + B_n$, where
    \begin{align*}
        A_n
        & = (\xi + 1) \Bigl( \cL_{\eps f_{n-1}} \cdots \cL_{\eps f_0} \frac{h + \xi g}{\xi + 1}
        - \cL_{\eps f_{n-1}} \cdots \cL_{\eps f_0} g \Bigr)
        , \\
        B_n
        & = \cL_{\eps f_{n-1}} \cdots \cL_{\eps f_0} g - \cL_{\eps g_{n-1}} \cdots \cL_{\eps g_0} g
        \\
        & = \sum_{j=0}^{n-1} \cL_{\eps f_{n-1}} \cdots \cL_{\eps f_{j+1}} (\cL_{\eps f_j} - \cL_{\eps g_j})
        \cL_{\eps g_{j-1}} \cdots \cL_{\eps g_0} g
        .
    \end{align*}
    By the invariance of $\cD^2_1$, assumption~\eqref{eq:Ldiffgamma} and Theorem~\ref{thm:Juho},
    \begin{align*}
        \|A_n\|_{L^1}
        & \leq C' (\xi + 2) n^{-1/\gamma + 1}
        ,\\
        \|B_n\|_{L^1}
        &\leq C' |\eps| \sum_{j=0}^{n-1} \| f_j - g_j \|_{L^1} (n-j)^{-1/\gamma + \beta / \gamma}
        .
    \end{align*}
    Here $C'$ depends only on $\bfC$.
    Let $\delta_n = \| f_n - g_n \|_{L^1}$. Then
    \[
        \delta_n
        \leq \|A_n\|_{L^1} + \|B_n\|_{L^1}
        \leq C' n^{ -1 / \gamma + 1}  + C' |\eps| \sum_{j=0}^{n-1} \delta_j ( n - j )^{ -1 / \gamma + \beta / \gamma}
        .
    \]
    By Lemma~\ref{lem:conv}, $\delta_n \leq \max \{2, C' (1 - |\eps| C' C_{\beta,\gamma})^{-1} \} n^{ -1 / \gamma + 1 }$ for all $n > 0$,
    provided that $|\eps| C' < C_{\beta, \gamma}^{-1}$.
\end{proof}

\begin{lem}
    \label{lem:inD}
    Suppose that $f$ is a probability density on $[0,1]$. For every $\delta > 0$ there exist $n \geq 0$
    and $g \in \cD_1^r$ such that $\| \cL_\eps^n f - g \|_{L^1} \leq \delta$.
\end{lem}

\begin{proof}
    Denote $f_k = \cL_\eps^k f$ and $\cL_k = \cL_{\eps f_0} \cdots \cL_{\eps f_{k-1}}$.
    Let $\tf$ be a $C^\infty$ probability density with $\|f - \tf\|_{L^1} \leq \delta / 2$.
    It exists because $C^\infty$ is dense in $L^1$.
    Then for all $k$,
    \[
        \| \cL_k f - \cL_k \tf \|_{L^1}
        \leq \| f - \tf \|_{L^1}
        \leq \delta / 2
        .
    \]
    Choose $C \geq 0$ large enough so that $(\tf + C) / (C + 1) \in \cD_1^r$. Write
    \[
        \cL_k \tf - \cL_k 1
        = (C + 1) \Bigl[ \cL_k \Bigl( \frac{\tf + C}{C+1} \Bigr) - \cL_k 1 \Bigr]
        .
    \]
    By Proposition~\ref{prop:mix}, the right hand side above converges to $0$, in particular
    $\|\cL_n \tf - \cL_n 1\|_{L^1} \leq \delta / 2$ for some $n$.

    Take $g = \cL_n 1$. Then $g \in \cD_1^r$ by the invariance of $\cD_1^r$,
    and $\| \cL_\eps^n f - g \| \leq \delta$ by construction.
\end{proof}

\begin{prop}
    \label{prop:physical}
    Suppose that $f$ is a probability density on $[0,1]$.
    There is $\eps_0 > 0$ such that for all $\eps \in [-\eps_0, \eps_0]$,
    \[
        \lim_{n \to \infty} \| \cL_\eps^n f - h_\eps \|_{L^1}
        = 0
        .
    \]
\end{prop}

\begin{proof}
    Choose a small $\delta > 0$.
    Without loss of generality, suppose that $\|f - \tf\|_{L^1} \leq \delta$ with $\tf \in \cD_1^1$.
    (The general case is recovered using Lemma~\ref{lem:inD} and replacing $f$ with $\cL_\eps^n f$
    with sufficiently large $n$.)

    As in the proof of Proposition~\ref{prop:mix}, denote $f_n = \cL_\eps^n f$ and $\tf_n = \cL_\eps^n \tf$,
    and write $f_n - \tf_n = A_n + B_n$, where
    \begin{align*}
        A_n
        & = \cL_{\eps f_{n-1}} \cdots \cL_{\eps f_0} f - \cL_{\eps f_{n-1}} \cdots \cL_{\eps f_0} \tf
        , \\
        B_n
        & = \cL_{\eps f_{n-1}} \cdots \cL_{\eps f_0} \tf - \cL_{\eps \tf_{n-1}} \cdots \cL_{\eps \tf_0} \tf
        \\
        & = \sum_{j=0}^{n-1} \cL_{\eps f_{n-1}}  \cdots \cL_{\eps f_{j+1}}
        (\cL_{\eps f_j} - \cL_{\eps \tf_j}) \cL_{\eps \tf_{j-1}} \cdots \cL_{\eps \tf_0} \tf
        .
    \end{align*}
    Since all $\cL_{\eps f_j}$ are contractions in $L^1$,
    \begin{equation}
        \label{eq:Afh}
        \|A_n\|_{L^1}
        \leq \|f - \tf\|_{L^1}
        \leq \delta
        .
    \end{equation}
    By~\eqref{eq:Ldiffgamma} and Theorem~\ref{thm:Juho},
    \begin{equation}
        \label{eq:Bfh}
        \begin{aligned}
            \|B_n\|_{L^1}
            & \leq C' |\eps| \sum_{j=0}^{n-1} \|f_j - \tf_j\|_{L^1} (n-j)^{-1/\gamma + \beta / \gamma}
            \\
            & \leq C'' |\eps| \max_{j < n} \|f_j - \tf_j\|_{L^1}
            ,
        \end{aligned}
    \end{equation}
    where $C'$ depends only on $\bfC$
    and $C'' = C' \sum_{j=1}^{\infty} j^{-1/\gamma + \beta / \gamma}$;
    recall that $-1/\gamma + \beta / \gamma < -1$, so this sum is finite.
    
    From~\eqref{eq:Afh} and~\eqref{eq:Bfh},
    \[
        \| f_n - \tf_n \|_{L^1}
        \leq \|A_n\|_{L^1} + \|B_n\|_{L^1}
        \leq \delta + C'' |\eps| \max_{j < n} \| f_j - \tf_j \|_{L^1}
        .
    \]
    Hence if $\eps$ is sufficiently small so that $C'' |\eps| < 1$, then
    \[
        \| f_n - \tf_n \|_{L^1} \leq \delta / (1 - C'' |\eps|)
        \quad \text{for all} \quad n
        .
    \]
    Since $\delta > 0$ is arbitrary, $\| f_n - \tf_n \|_{L^1} \to 0$ as $n \to \infty$.
\end{proof}

\section{Example: verification of assumptions}
\label{sec:example}

Here we verify that the example~\eqref{eq:Teh} fits the assumptions of Section~\ref{sec:main},
namely \ref{a:r}, \ref{a:gamma}, \ref{a:w}, \ref{a:d} and~\eqref{eq:Ldiffgamma}.
The key statements are Proposition~\ref{prop:abc} and Corollary~\ref{cor:partialL}.

Let $\eps_* > 0$ and denote $\gamma_- = \gamma_* - 2 \eps_*$ and $\gamma_+ = \gamma_* + 2 \eps_*$,
so that for all $\eps,h$,
\[
    \gamma_- - \eps_*
    < \gamma_* + \eps \gamma_h
    < \gamma_+ + \eps_*
    .
\]
Force $\eps_*$ to be small so that $0 < \gamma_- < \gamma_+ < 1$.
Let $\gamma = \gamma_+$ and fix $r \geq 2$.

In Proposition~\ref{prop:abc} we verify assumptions~\ref{a:r}, \ref{a:gamma}, \ref{a:w}, \ref{a:d}
from Section~\ref{sec:main}, and in Corollary~\ref{cor:partialL} we verify~\eqref{eq:Ldiffgamma}.

In this section we use the notation $A \lesssim B$ for $A \leq C B$ with $C$ depending only on $\eps_*$,
and $A \sim B$ for $A \lesssim B \lesssim A$.

\begin{prop}
    \label{prop:abc}
    The family of maps $T_{\eps h}$ satisfies assumptions~\ref{a:r}, \ref{a:gamma}, \ref{a:w}, \ref{a:d}
    from Section~\ref{sec:main}.
\end{prop}

\begin{proof}
    It is immediate that~\ref{a:r}, \ref{a:gamma} and~\ref{a:d} hold, so we only need to justify~\ref{a:w}.
    Denote $\tgamma = \gamma_* + \eps \gamma_h$ and observe that, with $w$ and each $w_{\ell, j}$ as in~\eqref{eq:bbb},
    \[
        \frac{1}{x^{\ell} \circ T_{\eps h}} - \frac{w^\ell}{x^\ell}
        \sim x^{\tgamma - \ell}
        \quad \text{and} \quad
        \frac{|w_{\ell,j}|}{x^j}
        \lesssim x^{\tgamma - \ell}
        .
    \]
    The implied constants depend on $\ell$ and $j$ but not on $\eps$ or $h$,
    and~\ref{a:w} follows.
\end{proof}

It remains to verify~\eqref{eq:Ldiffgamma}.
The precise expressions for $\gamma_h$ and $\varphi_h$ are not too important, we rely on the following
their properties:
\begin{itemize}
    \item $\varphi_h(0) = \varphi_h'(0) = \varphi_h(1) = 0$ for each $h$,
        so that, informally, $\varphi_h$ has no effect on the indifferent fixed point at $0$.
    \item The maps $h \mapsto \varphi_h$, $L^1 \to C^3$, $h \mapsto \varphi_h'$, $L^1 \mapsto C^2$,
        and $h \mapsto \gamma_h$, $L^1 \to \R$ are continuously 
        Fr\'echet differentiable, i.e.\ for each $h,f$,
        \begin{align*}
            \| \varphi_{h + f}  - \varphi_h  - \Phi_h  f \|_{C^3}
            & = o(\|f\|_{L^1})
            , \\
            \| \varphi'_{h + f} - \varphi'_h - \Phi'_h f \|_{C^2}
            & = o(\|f\|_{L^1})
            , \\
            | \gamma_{h + f}  - \gamma_h  - \Gamma_h  f |
            & = o(\|f\|_{L^1})
            .
        \end{align*}
        where $\Phi_h \colon L^1 \to C^3$, $\Phi'_h \colon L^1 \to C^2$ and $\Gamma \colon L^1 \to \R$
        are bounded linear operators, continuously depending on $h$.
\end{itemize}

Suppose that $f_0,f_1 \in L^1$ and $v \in \cD^2_1$
Let $f_s = (1-s) f_0 + s f_1$ with $s \in [0,1]$. Denote $T_s = T_{\eps f_s}$
and let $\cL_s$ be the associated transfer operator.

\begin{prop}
    \label{prop:partialL}
    $| \partial_s (\cL_s v) | \lesssim |\eps| x^{- (\gamma_+ - \gamma_-)}$ and
    $| (\partial_s (\cL_s v))'(x) | \lesssim |\eps| x^{- (\gamma_+ - \gamma_-) - 1}$.
\end{prop}

\begin{proof}
    We abuse notation, restricting to a single branch of $T_s$,
    so that $T_s$ is invertible and $\cL_s v = (v / T_s') \circ T_s^{-1}$.
    Let $\zeta_s = \Phi_{f_s} (f_1 - f_0)$, $\psi_s = \Phi'_{f_s}(f_1 - f_0)$
    and $\lambda_s = \Gamma_{f_s} (f_1 - f_0)$.
    Then
    \begin{align*}
        \partial_s (\cL_s v)
        & = \biggl[
            \frac{(v' T_s' - v T_s'') \partial_s T_s}{T_s'^3} + \frac{v \partial_s T_s'}{T_s'^2}
        \biggr] \circ T_s^{-1}
    \end{align*}
    with
    \begin{align*}
        (\partial_s T_s)(x)
        & = \eps \lambda_s x^{1 + \gamma + \eps \nu_{f_s}} \log x + \eps \zeta_s(x)
        , \\
        (\partial_s T_s')(x)
        & = \eps \lambda_s x^{\gamma + \eps \nu_{f_s}} [ 1 + (1 + \gamma + \eps \nu_{f_s}) \log x ] + \eps \psi_s(x)
        .
    \end{align*}
    Observe that $|v(x)| \lesssim x^{-\gamma_+}$, $|v'(x)| \lesssim x^{-\gamma_+ - 1}$,
    $|\partial_s T_s|  \lesssim |\eps| x^{1 + \gamma_-}$,
    $|\partial_s T_s'| \lesssim |\eps| x^{\gamma_-}$,
    $T_s'(x) \sim 1$ and $|T_s''(x)| \lesssim x^{\gamma_- - 1}$.
    Hence
    \[
        |\partial_s (\cL_s v) (x)|
        \lesssim |\eps| x^{-(\gamma_+ - \gamma_-)}
        .
    \]
    Differentiating in $x$ further and observing that $|v''(x)| \lesssim x^{-\gamma_+ - 2}$,
    $|T_s'''(x)| \lesssim x^{\gamma_- - 2}$,
    $|(\partial_s T_s)'(x)|  \lesssim |\eps| x^{\gamma_-}$ and
    $|(\partial_s T_s')'(x)| \lesssim |\eps| x^{\gamma_- - 1}$, we obtain
    \[
        |(\partial_s (\cL_s v))'(x)|
        \lesssim |\eps| x^{-(\gamma_+ - \gamma_-) - 1}
        .
    \]
\end{proof}

\begin{cor}
    \label{cor:partialL}
    In the setup of Proposition~\ref{prop:partialL}, we can represent
    \[
        \cL_{\eps f_0} v - \cL_{\eps f_1} v
        = \delta (g_0 - g_1)
        ,
    \]
    where $\delta \lesssim |\eps| \|f_0 - f_1\|_{L^1}$,
    and $g_0, g_1 \in \cD^1_1$ with $g_0(x), g_1(x) \lesssim x^{-4 \eps}$.
\end{cor}

\end{document}